\documentclass{amsart}
\usepackage{setspace}
\usepackage{a4}
\usepackage{amssymb,amsmath,amsthm,latexsym}
\usepackage{amsfonts}
\usepackage{amsfonts}
\usepackage{graphicx}
\usepackage{textcomp}
\usepackage{cite}
\usepackage{enumerate}
\usepackage[mathscr]{euscript}
\usepackage{mathtools}
\newtheorem{theorem}{Theorem}[section]

\newtheorem{conjecture}[theorem]{Conjecture}
\newtheorem{corollary}[theorem] {Corollary}
\newtheorem{definition}[theorem]{Definition}

\newtheorem{proposition}[theorem]{Proposition}
\newtheorem{remark}[theorem]{Remark}
\newtheorem{question}[theorem]{Question}
\title{This is the title}
\usepackage{amssymb}
\usepackage{amssymb}
\usepackage{amssymb}
\usepackage{amssymb}
\usepackage{amsmath}
\usepackage{tikz}
\usepackage{hyperref}
\usepackage{enumerate}
\usepackage{mathtools}
\usepackage{amsmath}
\usepackage{tikz}
\usepackage{amssymb}
\usepackage{amsmath}
\usepackage{tikz}
\usepackage{hyperref}
\usepackage{fancyhdr}
\begin{document}

	\begin{center}
		{\bf{MODULAR WELCH BOUNDS WITH APPLICATIONS}}\\
		\textbf{K. MAHESH KRISHNA}\\
Post Doctoral Fellow \\
Statistics and Mathematics Unit\\
		Indian Statistical Institute, Bangalore Centre\\
		Karnataka 560 059 India\\
		Email: kmaheshak@gmail.com \\
		\today
	\end{center}
	
	\hrule
	\vspace{0.5cm}

%--------------------------------------
\textbf{Abstract}:  We prove the following two results. 
\begin{enumerate}
	\item Let $\mathcal{A}$ be a unital commutative C*-algebra and $\mathcal{A}^d$ be the standard Hilbert C*-module over $\mathcal{A}$.  Let $n\geq d$. If 	$\{\tau_j\}_{j=1}^n$ is any collection of  vectors in $\mathcal{A}^d$ such that $\langle \tau_j, \tau_j \rangle =1$, $\forall 1\leq j \leq n$, then 
		\begin{align*}
		\max _{1\leq j,k \leq n, j\neq k}\|\langle \tau_j, \tau_k\rangle ||^{2m}\geq \frac{1}{n-1}\left[\frac{n}{{d+m-1\choose m}}-1\right], \quad \forall m \in \mathbb{N}.
	\end{align*}
	\item  Let $\mathcal{A}$ be a $\sigma$-finite commutative W*-algebra or a commutative  AW*-algebra and $\mathcal{E}$ be a rank d Hilbert C*-module over $\mathcal{A}$. Let $n\geq d$. If 	$\{\tau_j\}_{j=1}^n$ is any collection of  vectors in $\mathcal{E}$ such that $\langle \tau_j, \tau_j \rangle =1$, $\forall 1\leq j \leq n$, then 
	\begin{align*}
		\max _{1\leq j,k \leq n, j\neq k}\|\langle \tau_j, \tau_k\rangle ||^{2m}\geq \frac{1}{n-1}\left[\frac{n}{{d+m-1\choose m}}-1\right], \quad \forall m \in \mathbb{N}.
	\end{align*}
\end{enumerate}
Results (1) and (2) reduce to the famous result of  Welch [\textit{IEEE Transactions on  Information Theory, 1974}] obtained 48 years ago.  We introduce the notions of modular frame potential, modular equiangular frames and modular  Grassmannian frames. We formulate Zauner's conjecture for Hilbert C*-modules.



\hrule
\tableofcontents

\section{Introduction}
In the study of cross correlations of signals,  L. Welch proved the following result in 1974.
\begin{theorem}\cite{WELCH}\label{WELCHTHEOREM} (\textbf{Welch bounds})
Let $n\geq d$.	If	$\{\tau_j\}_{j=1}^n$  is any collection of  unit vectors in $\mathbb{C}^d$, then
	\begin{align*}
		\sum_{j=1}^n\sum_{k=1}^n|\langle \tau_j, \tau_k\rangle |^{2m}\geq \frac{n^2}{{d+m-1\choose m}}, \quad \forall m \in \mathbb{N}.
	\end{align*}
	In particular,
	
	\begin{align*}
		\sum_{j=1}^n\sum_{k=1}^n|\langle \tau_j, \tau_k\rangle |^{2}\geq \frac{n^2}{{d}}.
	\end{align*}
Further, 
	\begin{align}\label{FIRST123}
\text{(\textbf{Higher order Welch bounds})}	\quad		\max _{1\leq j,k \leq n, j\neq k}|\langle \tau_j, \tau_k\rangle |^{2m}\geq \frac{1}{n-1}\left[\frac{n}{{d+m-1\choose m}}-1\right], \quad \forall m \in \mathbb{N}.
	\end{align}
	In particular,
	\begin{align*}
	\text{(\textbf{First order Welch bound})}\quad 	\max _{1\leq j,k \leq n, j\neq k}|\langle \tau_j, \tau_k\rangle |^{2}\geq\frac{n-d}{d(n-1)}.
	\end{align*}
\end{theorem}
Proof of Theorem  \ref{WELCHTHEOREM} (mainly the first order Welch bound) can be found at various places, namely \cite{MASSEYON, DATTAHOWARD, MASSEYMITTELHOLZER, ROSENFELD, SUSTIKTROPP, WALDONBOOK, WALDRON2003, WALDRONSH}. It also appears implicitly in the celebrated paper \cite{CONWAYHARDINSLOANE} (see \cite{CALDERBANK, FICKUSJASPERMIXON}). After about half century from the  work of Welch, now Welch bounds are used in various places such as root-mean-square (RMS) absolute cross relation of unit vectors \cite{SARWATEBOOK}, frame potential \cite{BENEDETTOFICKUS, CASAZZAFICKUSOTHERS, BODMANNHAASPOTENTIAL},  correlations \cite{SARWATE},  codebooks \cite{DINGFENG}, numerical search algorithms  \cite{XIA, XIACORRECTION}, quantum measurements 
\cite{SCOTTTIGHT}, coding and communications \cite{TROPPDHILLON, STROHMERHEATH}, code division multiple access (CDMA) systems \cite{CHEBIRA1, CHEBIRA2}, wireless systems \cite{YATES}, compressed sensing \cite{TAN}, `game of Sloanes' \cite{JASPERKINGMIXON}, equiangular tight frames \cite{SUSTIKTROPP},  etc.

Over the time, Theorem \ref{WELCHTHEOREM} has been generalized for dual frames, vectors which may not be normalized, continuous frames indexed by complex projective space,  continuous Bessel sequences indexed by arbitrary measure spaces, natural number $m$ replaced by arbitrary reals \cite{CHRISTENSENDATTAKIM, DATTAWELCHLMA, WALDRONSH, HAIKINZAMIRGAVISH, EHLEROKOUDJOU, WALDONBOOK, MAHESHKRISHNA, WALDRON2003, DATTAHOWARD}. But to the best of our knowledge, Theorem  \ref{WELCHTHEOREM}  is not known for Hilbert C*-modules. Hilbert C*-modules are first introduced by Kaplansky \cite{KAPLANSKY2} for commutative C*-algebras and later developed for modules over arbitrary C*-algebras by Paschke  \cite{PASCHKE} and Rieffel \cite{RIEFFEL}. Revolutionary work of Kasparov made the theory as a powerful tool in K-theory \cite{KASPAROV, BLACKADAR, JENSENTHOMSEN}. We refer \cite{LANCE, MANUILOVTROITSKY, WEGGEOLSEN, MINGOPHILLIPS, FRANKPO, DUPREFILLMORE} for further information on Hilbert C*-modules. We remark here that the action of C*-algebra is from the left and the inner product is linear in first variable and *-linear in second variable.

%define potential and connect with trace
\section{Modular Welch bounds}
We first derive Welch bound for standard Hilbert C*-modules over unital commutative C*-algebras. First we need a definition.
\begin{definition}
	A collection $\{\tau_j\}_{j=1}^n$ in a Hilbert C*-module $\mathcal{E}$  over unital C*-algebra $\mathcal{A}$ with identity $1$ is said to have \textbf{unit inner product} if
	\begin{align*}
		\langle \tau_j, \tau_j \rangle =1, \quad \forall 1\leq j \leq n.
	\end{align*}
\end{definition}
It is clear that unit inner product vectors have unit norm but not the converse.
\begin{theorem}\textbf{(Modular Welch bounds)}\label{STANDARDMODULE}
 Let $\mathcal{A}$ be a unital commutative C*-algebra and $\mathcal{A}^d$ be the standard Hilbert C*-module over  $\mathcal{A}$.  Let $n\geq d$. If 	$\{\tau_j\}_{j=1}^n$ is any collection of unit inner product  vectors in $\mathcal{A}^d$, then 
 	\begin{align*}
  \sum_{j=1}^n\sum_{k=1}^n\|\langle \tau_j, \tau_k\rangle \|^{2m}	\geq \sum_{j=1}^n\sum_{k=1}^n\langle \tau_j, \tau_k\rangle ^{m}\langle \tau_k, \tau_j\rangle ^{m}\geq \frac{n^2}{{d+m-1\choose m}}, \quad \forall m \in \mathbb{N}.
 \end{align*}
 In particular,
 
 \begin{align*}
\sum_{j=1}^n\sum_{k=1}^n\|\langle \tau_j, \tau_k\rangle \|^{2} \geq 	\sum_{j=1}^n\sum_{k=1}^n\langle \tau_j, \tau_k\rangle \langle \tau_k, \tau_j\rangle \geq \frac{n^2}{{d}}.
 \end{align*}
 Further, 
 \begin{align*}
 	\text{(\textbf{Higher order modular Welch bounds})}	\quad		\max _{1\leq j,k \leq n, j\neq k}\|\langle \tau_j, \tau_k\rangle \|^{2m}\geq \frac{1}{n-1}\left[\frac{n}{{d+m-1\choose m}}-1\right], \quad \forall m \in \mathbb{N}.
 \end{align*}
 In particular,
 \begin{align*}
 	\text{(\textbf{First order modular Welch bound})}\quad 	\max _{1\leq j,k \leq n, j\neq k}\|\langle \tau_j, \tau_k\rangle \|^{2}\geq\frac{n-d}{d(n-1)}.
 \end{align*}	
\end{theorem}
\begin{proof}
Our proof is motivated from the proof of Welch \cite{WELCH}. 	Let $\tau_j\coloneqq(a_1^{(j)}, a_2^{(j)}, \dots, a_d^{(j)}) $ for all $1\leq j\leq n$. Define 
\begin{align*}
	B_m\coloneqq \sum_{j=1}^n\sum_{k=1}^n\langle \tau_j, \tau_k\rangle ^{m}\langle \tau_k, \tau_j\rangle ^{m}.
\end{align*}
Then using C*-algebra identity, 
\begin{align*}
	B_m&= \sum_{1\leq j, k \leq n, j\neq k}\langle \tau_j, \tau_k\rangle ^{m}\langle \tau_k, \tau_j\rangle ^{m}+\sum_{j=1}^n\langle \tau_j, \tau_j\rangle ^{2m}\\
	&\leq  \sum_{1\leq j, k \leq n, j\neq k}\|\langle \tau_j, \tau_k\rangle^m \langle \tau_k, \tau_j\rangle ^{m}\|+\sum_{j=1}^n\|\langle \tau_j, \tau_j\rangle ^{2m}\|\\
	&=\sum_{1\leq j, k \leq n, j\neq k}\|\langle \tau_j, \tau_k\rangle^m \|^2+\sum_{j=1}^n\|\langle \tau_j, \tau_j\rangle \|^{2m}\\
	&=\sum_{1\leq j, k \leq n, j\neq k}\|\langle \tau_j, \tau_k\rangle^m \|^2+n\\
	&\leq \sum_{1\leq j, k \leq n, j\neq k}\|\langle \tau_j, \tau_k\rangle \|^{2m}+n\\
	&\leq (n^2-n)	\max _{1\leq j,k \leq n, j\neq k}\|\langle \tau_j, \tau_k\rangle \|^{2m}+n.
\end{align*}
On the other hand, using the definition of standard inner product and Cauchy-Schwarz inequality in Hilbert C*-modules,

\begin{align*}
B_m&= \sum_{j=1}^n\sum_{k=1}^n\left(\sum_{r=1}^{d}a_r^{(j)}(a_r^{(k)})^*\right)	\left(\sum_{s=1}^{d}a_s^{(k)}(a_s^{(j)})^*\right)	\\
&= \sum_{j=1}^n\sum_{k=1}^n\sum_{u_1\cdots u_m=1, v_1\cdots v_m=1}^{d}\prod_{p=1}^{m}a_{u_p}^{(j)}(a_{v_p}^{(j)})^*(a_{u_p}^{(k)})^*a_{v_p}^{(k)}\\
&=\sum_{j=1}^n\sum_{k=1}^n\sum_{u_1\cdots u_m=1, v_1\cdots v_m=1}^{d}\left(\prod_{p=1}^{m}a_{u_p}^{(j)}(a_{v_p}^{(j)})^*\right)\left(\prod_{q=1}^{m}a_{u_q}^{(j)}(a_{v_q}^{(j)})^*\right)^*\\
&=\sum_{x_1\cdots x_d, y_1\cdots y_d, \sum_{r=1}^{d}x_r=\sum_{s=1}^{d}y_s=m} {m \choose x_1, \dots, x_d}{m \choose y_1, \dots, y_d}\times\\
&\quad \left(\sum_{j=1}^{n}\prod_{r=1}^{d}(a_r^{(j)})^{x_r}((a_r^{(j)})^*)^{y_r}\right)\left(\sum_{k=1}^{n}\prod_{s=1}^{d}(a_s^{(k)})^{x_s}((a_s^{(k)})^*)^{y_s}\right)^*\\
&\geq \sum_{x_1\cdots x_d, \sum_{r=1}^{d}x_r=m} \left({m \choose x_1, \dots, x_d}\sum_{j=1}^{n}\prod_{r=1}^{d}(a_r^{(j)})^{x_r}((a_r^{(j)})^*)^{x_r}\right)\times\\
& \quad \left({m \choose x_1, \dots, x_d}\sum_{t=1}^{n}\prod_{l=1}^{d}(a_l^{(t)})^{x_l}((a_l^{(t)})^*)^{x_l}\right)^*\\
&\geq  \frac{1}{\left\|\sum_{x_1\cdots x_d, \sum_{r=1}^{d}x_r=m}\right\|}\left(\sum_{x_1\cdots x_d, \sum_{r=1}^{d}x_r=m} {m \choose x_1, \dots, x_d}\sum_{j=1}^{n}\prod_{r=1}^{d}(a_r^{(j)})^{x_r}((a_r^{(j)})^*)^{x_r}\right)\times \\
&\quad \left(\sum_{x_1\cdots x_d, \sum_{s=1}^{d}x_s=m} {m \choose x_1, \dots, x_d}\sum_{k=1}^{n}\prod_{s=1}^{d}(a_s^{(k)})^{x_s}((a_s^{(k)})^*)^{x_s}\right)^*\\
&= \frac{1}{{d+m-1\choose m}}\left(\sum_{x_1\cdots x_d, \sum_{r=1}^{d}x_r=m} {m \choose x_1, \dots, x_d}\sum_{j=1}^{n}\prod_{r=1}^{d}(a_r^{(j)})^{x_r}((a_r^{(j)})^*)^{x_r}\right)\times \\
&\quad \left(\sum_{x_1\cdots x_d, \sum_{s=1}^{d}x_s=m} {m \choose x_1, \dots, x_d}\sum_{k=1}^{n}\prod_{s=1}^{d}(a_s^{(k)})^{x_s}((a_s^{(k)})^*)^{x_s}\right)^*\\
&= \frac{1}{{d+m-1\choose m}}\left(\sum_{j=1}^{n}\sum_{x_1\cdots x_d, \sum_{r=1}^{d}x_r=m} {m \choose x_1, \dots, x_d}\prod_{r=1}^{d}(a_r^{(j)})^{x_r}((a_r^{(j)})^*)^{x_r}\right)\times \\
&\quad \left(\sum_{k=1}^{n}\sum_{x_1\cdots x_d, \sum_{s=1}^{d}x_s=m} {m \choose x_1, \dots, x_d}\prod_{s=1}^{d}(a_s^{(k)})^{x_s}((a_s^{(k)})^*)^{x_s}\right)^*\\
&= \frac{1}{{d+m-1\choose m}}\left(\sum_{j=1}^{n}\left(\sum_{p=1}^{d}a_p^{(j)}(a_p^{(j)})^*\right)^m\right)\left(\sum_{k=1}^{n}\left(\sum_{q=1}^{d}a_q^{(k)}(a_q^{(k)})^*\right)^m\right)^*\\
&=\frac{1}{{d+m-1\choose m}}\left(\sum_{j=1}^{n}\left(1\right)^m\right)\left(\sum_{k=1}^{n}\left(1\right)^m\right)^*=\frac{n^2}{{d+m-1\choose m}}.
\end{align*}
Hence 

\begin{align*}
	(n^2-n)	\max _{1\leq j,k \leq n, j\neq k}\|\langle \tau_j, \tau_k\rangle \|^{2m}+n&\geq   \sum_{j=1}^n\sum_{k=1}^n\|\langle \tau_j, \tau_k\rangle \|^{2m}	\\
	&\geq \sum_{j=1}^n\sum_{k=1}^n\langle \tau_j, \tau_k\rangle ^{m}\langle \tau_k, \tau_j\rangle ^{m}\\
	&\geq  \frac{n^2}{{d+m-1\choose m}}.
\end{align*}
\end{proof}
\begin{corollary}
	Theorem  \ref{WELCHTHEOREM}    is a corollary of Theorem 	\ref{STANDARDMODULE}.
\end{corollary}
Note that proof of Theorem \ref{STANDARDMODULE}  works for the standard Hilbert C*-module w.r.t. standard inner product. We now try to get the result for arbitrary Hilbert C*-module. For this, we need some results and concepts. 

It is well known that trace of an operator can be determined using orthonormal basis and also from frames in Hilbert spaces (see \cite{HANKORNELSONLARSON}). We now try to derive such results for Hilbert C*-modules.
\begin{theorem}\label{TRACEFU}
	Let $\mathcal{A}$ be a commutative 	unital C*-algebra  and $\mathcal{E}$ be a finite rank (of rank $d$) Hilbert C*-module over $\mathcal{A}$. If $T: \mathcal{E} \to\mathcal{E}$ is an adjointable homomorphism, then for any two (and hence for all) orthonormal bases $\{\tau_j\}_{j=1}^d$ and $\{\omega_j\}_{j=1}^d$ for $\mathcal{E}$, 
	\begin{align*}
		\sum_{j=1}^{d}\langle T\tau_j, \tau_j \rangle =	\sum_{j=1}^{d}\langle T\omega_j, \omega_j \rangle.
	\end{align*}
\end{theorem}
\begin{proof}
	Note that commutative C*-algebras have invariant basis number  property \cite{GIPSON} (see \cite{LEAVITT6, LEAVITT5, LEAVITT4, LEAVITT3, LEAVITT2, LEAVITT} for more on invariant basis number property for rings). Hence any two bases have same number of elements.	By using the commutativity of the C*-algebra, we get 
	\begin{align*}
		\sum_{j=1}^{d}\langle T\tau_j, \tau_j \rangle &=	\sum_{j=1}^{d} \left\langle T\tau_j, \sum_{k=1}^{d}\langle \tau_j, \omega_k \rangle \omega_k \right \rangle =\sum_{j=1}^{d}\sum_{k=1}^{d} \langle T\tau_j, \omega_k \rangle \langle \omega_k, \tau_j \rangle \\
		&=\sum_{k=1}^{d}\sum_{j=1}^{d} \langle T\tau_j, \omega_k \rangle \langle \omega_k, \tau_j \rangle =\sum_{k=1}^{d} \left \langle \sum_{j=1}^n \langle \omega_k, \tau_j \rangle \tau_j, T^*\omega_k \right \rangle \\
		&=\sum_{k=1}^{d} \left \langle \omega_k, T^*\omega_k \right \rangle=	\sum_{k=1}^{d}\langle T\omega_k, \omega_k \rangle.
	\end{align*}
\end{proof}
Theorem  \ref{TRACEFU}  makes the following definition meaningful.
\begin{definition}
	Let $\mathcal{A}$ be a commutative 	unital C*-algebra  and $\mathcal{E}$ be a finite rank (of rank $d$) Hilbert C*-module over $\mathcal{A}$. If $T: \mathcal{E} \to\mathcal{E}$ is an adjointable homomorphism, then the \textbf{trace} of $T$ is defined as 	
	\begin{align*}
		\text{Tra}(T)\coloneqq 	\sum_{j=1}^{d}\langle T\tau_j, \tau_j \rangle,
	\end{align*}
for some 	orthonormal basis $\{\tau_j\}_{j=1}^d$  for $\mathcal{E}$.
\end{definition}
\begin{definition}
	Given a collection  $\{\tau_j\}_{j=1}^n$ in a Hilbert C*-module, the \textbf{frame homomorphism} $S_\tau: \mathcal{E} \to \mathcal{E}$ is defined as 
	
	\begin{align*}
		S_\tau x \coloneqq \sum_{j=1}^{n}\langle x, \tau_j \rangle \tau_j, \quad \forall x \in \mathcal{E}.
	\end{align*}
\end{definition}
Next result shows that trace of frame homomorphism can be recovered from its defining elements.
\begin{theorem}\label{TRACETHEOREM}
		Let $\mathcal{A}$ be a commutative 	unital C*-algebra  and $\mathcal{E}$ be a $d$-rank  Hilbert C*-module over $\mathcal{A}$. If $\{\tau_j\}_{j=1}^n$  is any collection in $\mathcal{E}$, then 
			\begin{align*}
			&	\text{Tra}(S_\tau)=\sum_{j=1}^n\langle \tau_j, \tau_j\rangle,\\
			&	\text{Tra}(S_\tau^2)=	\sum_{j=1}^n\sum_{k=1}^n\langle \tau_j, \tau_k\rangle \langle \tau_k, \tau_j\rangle.
		\end{align*}
\end{theorem}
\begin{proof}
	Let $\{\omega_j\}_{j=1}^d$ be an orthonormal basis for $\mathcal{E}$. Then	
\begin{align*}
		\text{Tra}(S_\tau)&=\sum_{j=1}^d\langle S_\tau\omega_j, \omega_j \rangle=\sum_{j=1}^d\left \langle \sum_{k=1}^n \langle \omega_j, \tau_k \rangle \tau_k, \omega_j\right  \rangle =\sum_{j=1}^d\sum_{k=1}^n \langle \omega_j, \tau_k \rangle \langle \tau_k, \omega_j \rangle \\
	&=\sum_{k=1}^n\sum_{j=1}^d \langle \omega_j, \tau_k \rangle \langle \tau_k, \omega_j \rangle=\sum_{k=1}^n\left\langle \sum_{j=1}^d\langle \tau_k,\omega_j \rangle\omega_j, \tau_k \right\rangle =\sum_{k=1}^n\langle \tau_k, \tau_k\rangle
\end{align*}
and 
\begin{align*}
\text{Tra}(S_\tau^2)&=	\sum_{j=1}^d\langle S_\tau^2\omega_j, \omega_j \rangle=\sum_{j=1}^d\langle S_\tau\omega_j,  S_\tau\omega_j \rangle=\sum_{j=1}^d\left \langle \sum_{k=1}^n \langle \omega_j, \tau_k \rangle \tau_k, S_\tau\omega_j\right  \rangle \\
&=\sum_{j=1}^d\sum_{k=1}^{n}\langle \omega_j, \tau_k \rangle\langle \tau_k,S_\tau\omega_j \rangle=\sum_{k=1}^{n}\left\langle \sum_{j=1}^d\langle \tau_k,S_\tau\omega_j \rangle\omega_j, \tau_k \right\rangle \\
&=\sum_{k=1}^{n}\left\langle \sum_{j=1}^d\langle S_\tau^*\tau_k,\omega_j \rangle\omega_j, \tau_k \right\rangle =\sum_{k=1}^{n} \langle S_\tau^*\tau_k, \tau_k\rangle \\
&=\sum_{k=1}^{n} \langle S_\tau\tau_k, \tau_k\rangle =\sum_{k=1}^{n}\left\langle \sum_{j=1}^{n}\left\langle \tau_k,\tau_j\right\rangle \tau_j,\tau_k\right\rangle =\sum_{k=1}^{n}\sum_{j=1}^{n}\langle \tau_k,\tau_j\rangle \langle \tau_j,\tau_k\rangle. 
\end{align*}
\end{proof}
For vector spaces we have the following result.
\begin{theorem}\cite{COMON, BOCCI}\label{SYMMETRICTENSORDIMENSION}
	If $\mathcal{V}$ is a vector space of dimension $d$ and $\text{Sym}^m(\mathcal{V})$ denotes the vector space of symmetric m-tensors, then 
	\begin{align*}
		\text{dim}(\text{Sym}^m(\mathcal{V}))={d+m-1 \choose m}, \quad \forall m \in \mathbb{N}.
	\end{align*}
\end{theorem}
If a ring has invariant basis number property, we can invoke the proof of Theorem \ref{SYMMETRICTENSORDIMENSION}   easily for modules and get the following result. 
\begin{theorem}\label{SYMMETRICTENSORDIMENSIONMODULE}
Let $\mathcal{R}$ be a ring with invariant basis number property.	If $\mathcal{M}$ is a module  of rank $d$ and $\text{Sym}^m(\mathcal{M})$ denotes the module  of symmetric m-tensors, then 
	\begin{align*}
		\text{rank}(\text{Sym}^m(\mathcal{M}))={d+m-1 \choose m}, \quad \forall m \in \mathbb{N}.
	\end{align*}
\end{theorem}

\begin{definition}\cite{MANUILOV}
	A W*-algebra is called 	\textbf{$\sigma$-finite} if it contains no more than a countable set of mutually orthogonal projections.
\end{definition}
\begin{definition}\cite{KAPLANSKYANNALS}
	A C*-algebra $\mathcal{A}$ is called an \textbf{AW*-algebra} if the following conditions hold.
	\begin{enumerate}[\upshape(i)]
		\item Any set of orthogonal projections has  supremum.
		\item Any maximal commutative self-adjoint  subalgebra of  $\mathcal{A}$ is generated by its projections. 
	\end{enumerate}
\end{definition}
\begin{theorem} \cite{MANUILOV, HEUNENREYES}\label{SPECTRALTHEOREM}
	(\textbf{Spectral theorem for Hilbert C*-modules}) Let $\mathcal{A}$ be a $\sigma$-finite W*-algebra or an AW*-algebra and let  $\mathcal{E}$ be a finite rank Hilbert C*-module over $\mathcal{A}$.  If  $T: \mathcal{E} \to\mathcal{E}$ is a normal adjointable homomorphism, then it can be diagonalized (by a unitary homomorphism).
\end{theorem}

\begin{theorem}
\textbf{(Modular Welch bounds)}\label{ARBITMODULE}
Let $\mathcal{A}$ be a commutative $\sigma$-finite W*-algebra or a commutative  AW*-algebra. Let $\mathcal{E}$ be  a Hilbert C*-module over $\mathcal{A}$ of rank $d$.  Let $n\geq d$. If 	$\{\tau_j\}_{j=1}^n$ is any collection of unit inner product  vectors in $\mathcal{E}$, then 
\begin{align*}
	\sum_{j=1}^n\sum_{k=1}^n\|\langle \tau_j, \tau_k\rangle \|^{2m}	\geq \sum_{j=1}^n\sum_{k=1}^n\langle \tau_j, \tau_k\rangle ^{m}\langle \tau_k, \tau_j\rangle ^{m}\geq \frac{n^2}{{d+m-1\choose m}}, \quad \forall m \in \mathbb{N}.
\end{align*}
In particular,
\begin{align*}
	\sum_{j=1}^n\sum_{k=1}^n\|\langle \tau_j, \tau_k\rangle \|^{2} \geq 	\sum_{j=1}^n\sum_{k=1}^n\langle \tau_j, \tau_k\rangle \langle \tau_k, \tau_j\rangle \geq \frac{n^2}{{d}}.
\end{align*}
Further, 
\begin{align*}
	\text{(\textbf{Higher order modular Welch bounds})}	\quad		\max _{1\leq j,k \leq n, j\neq k}\|\langle \tau_j, \tau_k\rangle \|^{2m}\geq \frac{1}{n-1}\left[\frac{n}{{d+m-1\choose m}}-1\right], \quad \forall m \in \mathbb{N}.
\end{align*}
In particular,
\begin{align*}
	\text{(\textbf{First order modular Welch bound})}\quad 	\max _{1\leq j,k \leq n, j\neq k}\|\langle \tau_j, \tau_k\rangle \|^{2}\geq\frac{n-d}{d(n-1)}.
\end{align*}		
\end{theorem}
\begin{proof}
	We first note that a W*-algebra as well as an AW*-algebra is always unital \cite{SAITO, SAKAI}. Next note that commutative C*-algebras have  invariant basis number property \cite{GIPSON} and hence the notion rank is meaningful. We consider the collection $\{\tau^{\otimes m}_j\}_{j=1}^n$ in $\text{Sym}^m(\mathcal{E})$. Theorem  \ref{SPECTRALTHEOREM}  says that the adjointable positive (self-adjoint) homomorphism 
	\begin{align*}
S_\tau:	\text{Sym}^m(\mathcal{E})\ni x \mapsto \sum_{j=1}^{n}	\langle x, \tau^{\otimes m}_j\rangle \tau^{\otimes m}_j \in \text{Sym}^m(\mathcal{E})
	\end{align*}
	can be diagonalized. Let $\lambda_1, \dots, \lambda_{\text{rank}(\text{Sym}^m(\mathcal{E}))}$ be the  eigenvalues	of $S_\tau$. Using Theorem \ref{TRACETHEOREM}, 	Theorem \ref{SYMMETRICTENSORDIMENSIONMODULE} and the Cauchy-Schwarz inequality in Hilbert C*-module, we then get 
	
\begin{align*}
	n^2&=\left(\sum_{j=1}^{n}\langle \tau_j, \tau_j \rangle ^{m}\right)^2=\left(\sum_{j=1}^{n}\langle \tau_j^{\otimes m}, \tau_j^{\otimes m} \rangle \right)^2=(\operatorname{Tra}(S_{\tau}))^2\\
	&=\left(\sum_{k=1}^{\text{rank}(\text{Sym}^m(\mathcal{E}))}
	\lambda_k\right)^2\leq \text{rank}(\text{Sym}^m(\mathcal{E})) \sum_{k=1}^{\text{rank}(\text{Sym}^m(\mathcal{E}))}
	\lambda_k^2\\
	&={d+m-1 \choose m}\operatorname{Tra}(S^2_{\tau})={d+m-1 \choose m}\sum_{j=1}^n\sum_{k=1}^n\langle \tau_j^{\otimes m}, \tau^{\otimes m}_k\rangle \langle \tau^{\otimes m}_k, \tau_j^{\otimes m}\rangle.\\
	&={d+m-1 \choose m}\sum_{j=1}^n\sum_{k=1}^n\langle \tau_j, \tau_k\rangle ^m\langle \tau_k, \tau_j\rangle^m
\end{align*}	
	and hence 
	\begin{align*}
	\frac{ n^2}{{d+m-1 \choose m}}	&=\sum_{j=1}^n\sum_{k=1}^n\langle \tau_j, \tau_k\rangle ^m\langle \tau_k, \tau_j\rangle^m=\sum_{1\leq j, k \leq n, j \neq k}\langle \tau_j, \tau_k\rangle ^m\langle \tau_k, \tau_j\rangle^m+\sum_{j=1}^n\langle \tau_j, \tau_j\rangle ^{2m}\\
	&=\sum_{1\leq j, k \leq n, j \neq k}\langle \tau_j, \tau_k\rangle ^m\langle \tau_k, \tau_j\rangle^m+n\leq \sum_{1\leq j, k \leq n, j \neq k}\|\langle \tau_j, \tau_k\rangle ^m\langle \tau_k, \tau_j\rangle^m\|+n\\
	&=\sum_{1\leq j, k \leq n, j \neq k}\|\langle \tau_j, \tau_k\rangle \|^{2m}+n\leq (n^2-n)	\max _{1\leq j,k \leq n, j\neq k}\|\langle \tau_j, \tau_k\rangle \|^{2m}+n
	\end{align*}
which gives the required inequality.
\end{proof}
We clearly  have the following corollary.
\begin{corollary}
	Theorem  \ref{WELCHTHEOREM}    is a corollary of Theorem 	\ref{ARBITMODULE}.
\end{corollary}
Since the spectral theorem fails for arbitrary Hilbert C*-modules \cite{KADISON, KADISON2, GROVEPEDERSEN} we can not do proof of Theorem \ref{ARBITMODULE} for Hilbert C*-modules over arbitrary unital C*-algebras. 

Following theorem  follows from a similar arguments used in the proof of Theorem \ref{ARBITMODULE}.
\begin{theorem}
Let $\mathcal{A}$ be a commutative $\sigma$-finite W*-algebra or a commutative  AW*-algebra. Let $\mathcal{E}$ be  a Hilbert C*-module over $\mathcal{A}$ of rank $d$.  Let $n\geq d$. If 	$\{\tau_j\}_{j=1}^n$ is any collection of unit inner product  vectors in $\mathcal{E}$, then 
\begin{align*}
	\sum_{j=1}^n\sum_{k=1}^n\|\langle \tau_j, \tau_k\rangle \|^{2m}	\geq \sum_{j=1}^n\sum_{k=1}^n\langle \tau_j, \tau_k\rangle ^{m}\langle \tau_k, \tau_j\rangle ^{m}\geq \frac{1}{{d+m-1\choose m}}\left(\sum_{j=1}^{n}\langle \tau_j, \tau_j \rangle ^{m}\right)^2, \quad \forall m \in \mathbb{N}.
\end{align*}
In particular,
\begin{align*}
	\sum_{j=1}^n\sum_{k=1}^n\|\langle \tau_j, \tau_k\rangle \|^{2} \geq 	\sum_{j=1}^n\sum_{k=1}^n\langle \tau_j, \tau_k\rangle \langle \tau_k, \tau_j\rangle \geq \frac{1}{{d}}\left(\sum_{j=1}^{n}\langle \tau_j, \tau_j \rangle ^{m}\right)^2.
\end{align*}
Further,
\begin{align*}
		\max _{1\leq j,k \leq n, j\neq k}\|\langle \tau_j, \tau_k\rangle \|^{2m} \geq \frac{1}{n^2-n}\left[\frac{1}{{d+m-1\choose m}}\left(\sum_{j=1}^{n}\langle \tau_j, \tau_j \rangle ^{m}\right)^2-\sum_{j=1}^{n}\|\tau_j\|^{4m}\right], \quad \forall m \in \mathbb{N}.
\end{align*}
In particular, 
\begin{align*}
	\max _{1\leq j,k \leq n, j\neq k}\|\langle \tau_j, \tau_k\rangle \|^{2} \geq \frac{1}{n^2-n}\left[\frac{1}{d}\left(\sum_{j=1}^{n}\langle \tau_j, \tau_j \rangle \right)^2-\sum_{j=1}^{n}\|\tau_j\|^{4}\right].
\end{align*}
\end{theorem}
We now note that there are four more bounds along with Welch bounds. First,  we recall  a definition.
\begin{definition}\cite{JASPERKINGMIXON}
Given $d\in \mathbb{N}$, define \textbf{Gerzon's bound}
\begin{align*}
	\mathcal{Z}(d, \mathbb{K})\coloneqq 
	\left\{ \begin{array}{cc} 
		d^2 & \quad \text{if} \quad \mathbb{K} =\mathbb{C}\\
	\frac{d(d+1)}{2} & \quad \text{if} \quad \mathbb{K} =\mathbb{R}.\\
	\end{array} \right.
\end{align*}	
\end{definition}
\begin{theorem}\cite{JASPERKINGMIXON, XIACORRECTION, MUKKAVILLISABHAWALERKIPAAZHANG, SOLTANALIAN, BUKHCOX, CONWAYHARDINSLOANE, HAASHAMMENMIXON, RANKIN}  \label{LEVENSTEINBOUND}
Define $m\coloneqq \operatorname{dim}_{\mathbb{R}}(\mathbb{K})/2$.	If	$\{\tau_j\}_{j=1}^n$  is any collection of  unit vectors in $\mathbb{K}^d$, then
\begin{enumerate}[\upshape(i)]
	\item (\textbf{Bukh-Cox bound})
	\begin{align*}
		\max _{1\leq j,k \leq n, j\neq k}|\langle \tau_j, \tau_k\rangle |\geq \frac{\mathcal{Z}(n-d, \mathbb{K})}{n(1+m(n-d-1)\sqrt{m^{-1}+n-d})-\mathcal{Z}(n-d, \mathbb{K})}\quad \text{if} \quad n>d.
	\end{align*}
	\item (\textbf{Orthoplex/Rankin bound})	
	\begin{align*}
		\max _{1\leq j,k \leq n, j\neq k}|\langle \tau_j, \tau_k\rangle |\geq\frac{1}{\sqrt{d}} \quad \text{if} \quad n>\mathcal{Z}(d, \mathbb{K}).
	\end{align*}
	\item (\textbf{Levenstein bound})	
	\begin{align*}
		\max _{1\leq j,k \leq n, j\neq k}|\langle \tau_j, \tau_k\rangle |\geq \sqrt{\frac{n(m+1)-d(md+1)}{(n-d)(md+1)}} \quad \text{if} \quad n>\mathcal{Z}(d, \mathbb{K}).
	\end{align*}
\item (\textbf{Exponential bound})
	\begin{align*}
	\max _{1\leq j,k \leq n, j\neq k}|\langle \tau_j, \tau_k\rangle |\geq 1-2n^{\frac{-1}{d-1}}.
\end{align*}
\end{enumerate}	
\end{theorem}
Theorem \ref{LEVENSTEINBOUND}   leads to the following problem.
\begin{question}
\textbf{Whether there is a  modular  version of Theorem \ref{LEVENSTEINBOUND}?}. In particular, does there exists a  version of 
\begin{enumerate}[\upshape(i)]
	\item \textbf{Bukh-Cox bound for Hilbert C*-modules?}
	\item \textbf{Orthoplex/Rankin bound for Hilbert C*-modules?}
	\item \textbf{Levenstein bound for Hilbert C*-modules?}
	\item \textbf{Exponential bound for Hilbert C*-modules?}
\end{enumerate}		
\end{question}

\section{Applications}
 Throughout this section,  $\mathcal{A}$ is a  commutative $\sigma$-finite W*-algebra or a commutative  AW*-algebra and   $\mathcal{E}$ be a $d$-rank Hilbert C*-module over $\mathcal{A}$.
\begin{definition}\label{CRMSDEFINITION}
Let 	$\{\tau_j\}_{j=1}^n$ be a unit inner product collection  in   $\mathcal{E}$. We define  the \textbf{modular  root-mean-square} (MRMS)  cross relation of 	$\{\tau_j\}_{j=1}^n$   as 	
\begin{align*}
	I_{\text{MRMS}} (\{\tau_j\}_{j=1}^n)\coloneqq \left(\frac{1}{n(n-1)}\sum _{1\leq j,k \leq n, j\neq k}\langle \tau_j, \tau_k\rangle \langle \tau_k, \tau_j\rangle \right)^\frac{1}{2}.
\end{align*}
\end{definition}
Theorem  \ref{ARBITMODULE} now gives the following result.
\begin{proposition}
If 	$\{\tau_j\}_{j=1}^n$ is a unit inner product collection  in   $\mathcal{E}$, then 
	\begin{align*}
1\geq \|	I_{\text{MRMS}} (\{\tau_j\}_{j=1}^n)\|\geq 	I_{\text{MRMS}} (\{\tau_j\}_{j=1}^n)\geq \left(\frac{n-d}{d(n-1)}\right)^\frac{1}{2}.
	\end{align*}
\end{proposition}
\begin{definition}\label{CONTINUOUSPOTENTIALDEFINITION}
Let 	$\{\tau_j\}_{j=1}^n$ be a unit  inner product collection  in   $\mathcal{E}.$  The \textbf{modular  frame potential} of  	$\{\tau_j\}_{j=1}^n$ is defined as 
 \begin{align*}
 	MFP(\{\tau_j\}_{j=1}^n)\coloneqq	\sum_{j=1}^n\sum_{k=1}^n\langle \tau_j, \tau_k\rangle \langle \tau_k, \tau_j\rangle. 
 \end{align*}	
\end{definition}
Theorem  \ref{ARBITMODULE} again gives the following.
\begin{proposition}\label{FPESTIMATE}
If 	$\{\tau_j\}_{j=1}^n$ is a unit inner product collection  in   $\mathcal{E}$, then 
\begin{align*}
n^2\geq \|MFP(\{\tau_j\}_{j=1}^n)\|\geq 	MFP(\{\tau_j\}_{j=1}^n)\geq \frac{n^2}{{d}}.
\end{align*}	
\end{proposition}
%\begin{proof}
	%\end{proof}
	We now recall the definition of frames for Hilbert C*-modules \cite{FRANKLARSON, FRANKLARSON2} and ask a problem. For more on frame theory in Hilbert C*-modules and its remarkable connection with the development of Schatten class and trace class of operators on Hilbert C*-modules, we refer \cite{RAEBURNTHOMPSON, ARAMBASIC, LI, ASADIFRANK, ASADIFRANK2, JINGTHESIS, STERNSUIJLEKOM}. We remark here that we consider only finite frames for finite rank modules.  
	\begin{definition}\cite{FRANKLARSON, FRANKLARSON2}
	Let 	 $\mathcal{E}$ be a Hilbert C*-module over a unital C*-algebra $\mathcal{A}$. A collection $\{\tau_j\}_{j=1}^n$ in  $\mathcal{E}$ is said to be a (modular) \textbf{frame} for  $\mathcal{E}$ if there exist real $a,b>0$ such that 
	\begin{align*}
		a \langle x , x \rangle \leq \sum_{j=1}^n \langle x, \tau_j \rangle \langle  \tau_j, x \rangle  \leq b\langle x , x \rangle, \quad \forall x \in \mathcal{E}.
	\end{align*}
If $a=b$, then we say that the frame is tight and if $a=b=1$, then we say that the frame is Parseval.
	\end{definition}
\begin{question}
\textbf{Is  there  a characterization of  modular frames using modular  frame potential (like Theorem 7.1 in \cite{BENEDETTOFICKUS})?}
\end{question}
\begin{definition}
Let 	$\{\tau_j\}_{j=1}^n$ be a unit  inner product collection  in   $\mathcal{E}.$	We define the  \textbf{modular frame correlation} of 	$\{\tau_j\}_{j=1}^n$ as 
\begin{align*}
	\mathcal{M}(\{\tau_j\}_{j=1}^n)\coloneqq \max _{1\leq j,k \leq n, j\neq k}\|\langle \tau_j, \tau_k\rangle \|.
\end{align*}
\end{definition}

\begin{definition}
  A unit inner product   frame $\{\tau_j\}_{j=1}^n$ for  $\mathcal{E}$ is said to be a \textbf{modular Grassmannian frame} for  $\mathcal{E}$ if 
\begin{align*}
	\mathcal{M}(\{\tau_j\}_{j=1}^n)=\inf\left\{\mathcal{M}(\{\omega_j\}_{j=1}^n):\{\omega_j\}_{j=1}^n\text{ is a unit inner product  frame for }\mathcal{E} \right\}.	
\end{align*}		
\end{definition}
  It is known, using   compactness and continuity arguments that Grassmannian frames exist in every dimension in Hilbert spaces \cite{BENEDETTONKOLESAR}.  We can not use this argument for Hilbert C*-modules. Hence we have following problem. 
\begin{question}
	\textbf{Does every finite rank Hilbert C*-module  admits a modular Grassmannian frame?}
\end{question}
\begin{definition}
 A modular  frame $\{\tau_j\}_{j=1}^n$ for  $\mathcal{E}$ is said to be 	\textbf{$\gamma$-equiangular} if there exists a positive element  $\gamma\geq0$ in the C*-algebra such that
\begin{align*}
\langle\tau_j, \tau_k \rangle\langle\tau_k, \tau_j \rangle=\gamma, \quad \forall 	1\leq j , k \leq n , j\neq k.
\end{align*}
\end{definition}
We now formulate Zauner's conjecture for Hilbert C*-modules (see \cite{APPLEBY, ZAUNER, FUCHSHOANGSTACEY, RENESBLUME, KOPP, APPLEBY123, BENGTSSON2, MAGSINO} for the Zauner's conjecture in Hilbert spaces and its connections with Hilbert's 12-problem and Stark conjecture). 
\begin{conjecture}
	(\textbf{Modular Zauner's conjecture}) \textbf{Let $\mathcal{A}$ be a unital C*-algebra. For    every $d\in \mathbb{N}$, there exists a $\frac{1}{d+1}$-equiangular  unit inner product frame $\{\tau_j\}_{j=1}^{d^2}$  for $\mathcal{A}^d$, i.e., there exists  a tight frame  $\{\tau_j\}_{j=1}^{d^2}$  for  $\mathcal{A}^d$ satisfying $\langle \tau_j, \tau_j\rangle=1$, $\forall 1\leq j \leq n$, and 
	\begin{align*}
	\langle \tau_j, \tau_k\rangle	\langle\tau_k, \tau_j \rangle=\frac{1}{d+1}, \quad 	1\leq j , k \leq d^2 , j\neq k.
	\end{align*}}
\end{conjecture}
\begin{conjecture}
	(\textbf{Modular Zauner's conjecture- strong form}) \textbf{Let $\mathcal{A}$ be a unital C*-algebra with invariant basis number property or a W*-algebra. For    every $d\in \mathbb{N}$, there exists a $\frac{1}{d+1}$-equiangular  unit inner product frame $\{\tau_j\}_{j=1}^{d^2}$  for $\mathcal{A}^d$, i.e., there exists  a tight frame  $\{\tau_j\}_{j=1}^{d^2}$  for  $\mathcal{A}^d$ satisfying $\langle \tau_j, \tau_j\rangle=1$, $\forall 1\leq j \leq n$, and 
		\begin{align*}
			\langle \tau_j, \tau_k\rangle	\langle\tau_k, \tau_j \rangle=\frac{1}{d+1}, \quad 	1\leq j , k \leq d^2 , j\neq k.
	\end{align*}}
\end{conjecture}
\begin{theorem}\label{CNG2}
Let  $\{\tau_j\}_{j=1}^n$ be a unit inner product frame  for  $\mathcal{E}$. Then 
\begin{align}\label{EQUIANGULARINEQUALITY}
		\mathcal{M}(\{\tau_j\}_{j=1}^n)\geq \sqrt{\frac{n-d}{d(n-1)}}\eqqcolon\gamma.
\end{align}
If the frame is $\gamma$-equiangular, then we have equality in Inequality (\ref{EQUIANGULARINEQUALITY}).
\end{theorem}

\begin{remark}
\textbf{Using the Kasparov theory of integration in Hilbert C*-modules \cite{KASPAROVTO, TROITSKY}    over Lie groups \cite{BUMP, DUISTERMAAT}, we can derive results obtained in \cite{MAHESHKRISHNA} for Hilbert C*-modules (in a similar line) whenever the measure space is a Lie group and the collection is a modular continuous Bessel family defined as follows.}
\end{remark}
\begin{definition}\label{CONTINUOUSMODULARWELCH}
Let 	$(G, \mu_G)$ be a Lie group with (left) Haar measure. A collection   $\{\tau_g\}_{g \in G}$ in a Hilbert C*-module  $\mathcal{E}$ over $\mathcal{A}$ is said to be a \textbf{continuous modular frame} for $\mathcal{E}$ if the following conditions hold.
\begin{enumerate}[\upshape(i)]
	\item For each $x \in \mathcal{E}$, the map $G \ni g \mapsto \langle x, \tau_g \rangle \in \mathcal{A}$ is continuous.
	\item There are real $a,b>0$ such that 
	\begin{align*}
		a\langle x, x \rangle\leq \int_{G}\langle x, \tau_g \rangle \langle  \tau_g, x \rangle\,d\mu_G(g)\leq b \langle x, x \rangle , \quad \forall x \in \mathcal{E}.
	\end{align*}
\end{enumerate}
 If we do not demand the first inequality in (ii), then we say $\{\tau_g\}_{g \in G}$ is a \textbf{continuous modular Bessel family} for $\mathcal{E}$. 	
\end{definition}
Rather working with Definition \ref{CONTINUOUSMODULARWELCH}, we can also work with the following definition which is equivalent to Definition \ref{CONTINUOUSMODULARWELCH} (we can prove this using Theorem 2.8 in  \cite{PASCHKE}).
\begin{definition}
Let 	$(G, \mu_G)$ be a Lie group with (left) Haar measure. A collection   $\{\tau_g\}_{g \in G}$ in a Hilbert C*-module  $\mathcal{E}$ over $\mathcal{A}$ is said to be a \textbf{continuous modular frame} for $\mathcal{E}$ if the following conditions hold.
\begin{enumerate}[\upshape(i)]
	\item For each $x \in \mathcal{E}$, the map $G \ni g \mapsto \langle x, \tau_g \rangle \in \mathcal{A}$ is continuous.
	\item There are real $a,b>0$ such that 
	\begin{align*}
		a\|x\|^2\leq \left\|\int_{G}\langle x, \tau_g \rangle \langle  \tau_g, x \rangle\,d\mu_G(g)\right\|\leq b \|x\|^2, \quad \forall x \in \mathcal{E}.
	\end{align*}
\end{enumerate}
If we do not demand the first inequality in (ii), then we say $\{\tau_g\}_{g \in G}$ is a \textbf{continuous modular Bessel family} for $\mathcal{E}$. 		
\end{definition}
Here we list some sample  results, conjectures and concepts done    in \cite{MAHESHKRISHNA}  for Hilbert C*-modules. In the remaining part, $(G, \mu_G)$ is  a  Lie group with Haar measure.
\begin{theorem}(\textbf{Continuous modular Welch bounds})
Let $\mathcal{A}$ be a commutative $\sigma$-finite W*-algebra or a commutative  AW*-algebra. Let $\mathcal{E}$ be  a Hilbert C*-module over $\mathcal{A}$ of rank $d$. 		Let $(G, \mu_G)$ be a  Lie group  and $\{\tau_g\}_{g\in G}$ be a 	 continuous modular Bessel  family for $\mathcal{E}$ such that $\langle \tau_g, \tau_g \rangle =1, \forall g \in G$. If the diagonal $\Delta\coloneqq \{(g, g):g \in G\}$ is measurable in  $G\times G$, then 

	\begin{align*}
		\int_{G\times G}\|\langle \tau_g, \tau_h\rangle\|^{2m}\, d(\mu_G\times\mu_G)(g,h)&=	\int_{G}\int_{G}\|\langle \tau_g, \tau_h\rangle\|^{2m}\, d \mu_G(g)\, d \mu_G(h)\\
		&\geq	\int_{G\times G}\langle \tau_g, \tau_h\rangle^{m}\langle \tau_h, \tau_g\rangle^{m}\, d(\mu_G\times\mu_G)(g,h)\\
		&=	\int_{G}\int_{G}\langle \tau_g, \tau_h\rangle^{m}\langle \tau_h, \tau_g\rangle^{m}\, d \mu_G(g)\, d \mu_G(h)\geq \frac{\mu_G(G)^2}{{d+m-1\choose m}}, \quad \forall m \in \mathbb{N}.
	\end{align*}	
In particular, 
\begin{align*}
		\int_{G\times G}\|\langle \tau_g, \tau_h\rangle\|^{2}\, d(\mu_G\times\mu_G)(g,h)&=	\int_{G}\int_{G}\|\langle \tau_g, \tau_h\rangle\|^{2}\, d \mu_G(g)\, d \mu_G(h)\\
	&\geq	\int_{G\times G}\langle \tau_g, \tau_h\rangle\langle \tau_h, \tau_g\rangle\, d(\mu_G\times\mu_G)(g,h)\\
	&=	\int_{G}\int_{G}\langle \tau_g, \tau_h\rangle\langle \tau_h, \tau_g\rangle\, d \mu_G(g)\, d \mu_G(h)\geq \frac{\mu_G(G)^2}{d}.
\end{align*}
	Further, we have the \textbf{higher order continuous modular Welch bounds} 
	\begin{align*}
		\sup _{g, h \in G, g\neq h}\|\langle \tau_g, \tau_h\rangle \|^{2m}\geq \frac{1}{(\mu_G\times\mu_G)((G\times G)\setminus\Delta)}	\left[\frac{	\mu_G(G)^2}{{d+m-1 \choose m}}-(\mu_G\times\mu_G)(\Delta)\right],  \quad \forall m \in \mathbb{N}.
	\end{align*}
	In particular, we have the \textbf{first order continuous modular Welch bound} 
	\begin{align*}
		\sup _{g, h \in G, g\neq h}\|\langle \tau_g, \tau_h\rangle \|^{2}\geq \frac{1}{(\mu_G\times\mu_G)((G\times G)\setminus\Delta)}\left[\frac{\mu_G(G)^2}{d}-(\mu_G\times\mu_G)(\Delta)\right].
	\end{align*}
\end{theorem}
\begin{question}
	\textbf{Whether there is a  continuous modular  version of Theorem \ref{LEVENSTEINBOUND} for frames indexed by Lie groups}. In particular, does there exists a continuous  version of 
	\begin{enumerate}[\upshape(i)]
		\item \textbf{Bukh-Cox bound for Hilbert C*-modules?}
		\item \textbf{Orthoplex/Rankin bound for Hilbert C*-modules?}
		\item \textbf{Levenstein bound for Hilbert C*-modules?}
		\item \textbf{Exponential bound for Hilbert C*-modules?}
	\end{enumerate}		
\end{question}
\begin{definition}\label{12}
	Let  $\{\tau_g\}_{g\in G}$ be a 	unit inner product  continuous Bessel  family for  $\mathcal{E}$. If the diagonal $\Delta$ is measurable, then the \textbf{continuous modular root-mean-square} (CDRMS) absolute cross relation of $\{\tau_g\}_{g\in G}$  is defined as 	
	\begin{align*}
		I_{\text{CDRMS}}(\{\tau_g\}_{g\in G})\coloneqq	\left(\frac{1}{(\mu_G\times\mu_G)((G\times G)\setminus\Delta)}\int_{(G\times G)\setminus\Delta}\langle \tau_g, \tau_h\rangle\langle \tau_h, \tau_g\rangle\, d(\mu_G\times\mu_G)(g,h)\right)^\frac{1}{2}.
	\end{align*}
\end{definition}
\begin{proposition}
	Under the set up as in Definition \ref{12}, one has 
	\begin{align*}
		1\geq I_{\text{CDRMS}}(\{\tau_g\}_{g\in G})		\geq \left(\frac{1}{(\mu_G\times\mu_G)((G\times G)\setminus\Delta)}\left[\frac{\mu_G(  G)^2}{d}-(\mu_G\times\mu_G)(\Delta)\right]\right)^\frac{1}{2}.
	\end{align*}
\end{proposition}

\begin{definition}
	Let  	$\{\tau_g\}_{g\in G}$ be a 	unit inner product  continuous Bessel  family for  $\mathcal{E}$. The \textbf{continuous frame potential} of  	$\{\tau_g\}_{g\in G}$ is defined as 
	\begin{align*}
		CFP(\{\tau_g\}_{g\in G})\coloneqq \int_{G}\int_{G}\langle \tau_g, \tau_h\rangle \langle \tau_h, \tau_g\rangle\, d \mu_G(g)\, d \mu_G(h).
	\end{align*}	
\end{definition}
\begin{proposition}
	Given a unit inner product  continuous modular Bessel  family	$\{\tau_g\}_{g\in G}$ for $\mathcal{E}$, one has 
	\begin{align*}
		\frac{\mu_G(G)^2}{d}\leq FP(\{\tau_g\}_{g\in G})\leq \mu_G(G)^2.
	\end{align*}	
	Further, if the diagonal $\Delta$ is measurable, then one also has 
	\begin{align*}
		(\mu_G\times\mu_G)(\Delta)\leq CFP(\{\tau_g\}_{g\in G})\leq \mu_G(G)^2.
	\end{align*}
\end{proposition}
\begin{theorem}
	If  	$\{\tau_g\}_{g\in G}$ is a unit inner product continuous modular Bessel  family for  $\mathcal{E}$, then 
	\begin{align*}
		CFP(	\{\tau_g\}_{g\in G})=	\text{Tra}(S_\tau^2)=	\text{Tra}((\theta_\tau^*\theta_\tau)^2).
	\end{align*}	
\end{theorem}
\begin{definition}
	Let  	$\{\tau_g\}_{g\in G}$ be a 	unit inner product  continuous modular frame for $\mathcal{E}$.	We define the  \textbf{continuous frame correlation} of 	$\{\tau_g\}_{g\in G}$ as 
	\begin{align*}
		\mathcal{M}	\{\tau_g\}_{g\in G})\coloneqq \sup_{g,h \in G, g \neq h}\|\langle\tau_g, \tau_h \rangle\|.
	\end{align*}
\end{definition}
\begin{definition}
	A  continuous unit inner product modular frame 	$\{\tau_g\}_{g\in G}$ for   $\mathcal{E}$ is said to be a \textbf{continuous Grassmannian frame} for  $\mathcal{E}$ if 
	\begin{align*}
		\mathcal{M}(\{\tau_g\}_{g\in G})=\inf\left\{\mathcal{M}(\{\omega_g\}_{g\in G}):\{\omega_g\}_{g\in G}\text{ is a  continuous unit inner product modular frame for }\mathcal{E} \right\}.	
	\end{align*}		
\end{definition}
\begin{question}
	\textbf{Classify Lie groups  and (finite rank) Hilbert C*-modules so that continuous modular  Grassmannian frames exist}.
\end{question}

\begin{definition}
	A continuous modular frame 	$\{\tau_g\}_{g\in G}$ for  $\mathcal{E}$ is said to be 	\textbf{$\gamma$-equiangular} if there exists a positive element  $\gamma\geq0$ in the C*-algebra such that
	\begin{align*}
		\langle\tau_g, \tau_h \rangle\langle\tau_h, \tau_g \rangle=\gamma, \quad \forall 	g,h \in G, g \neq h.
	\end{align*}
\end{definition}

\begin{conjecture}
	(\textbf{Continuous modular Zauner's conjecture}) \textbf{For a given Lie group  $(G, \mu_G)$ and for  every $d\in \mathbb{N}$, there exists a $\gamma$-equiangular tight continuous modular frame 	$\{\tau_g\}_{g\in G}$  for $\mathcal{A}^d$ such that $\mu_G(G)=d^2$}.
\end{conjecture}
\begin{conjecture}
	(\textbf{Continuous modular Zauner's conjecture - strong form}) \textbf{Let $\mathcal{A}$ be a unital C*-algebra with invariant basis number property or a W*-algebra. For a given Lie group  $(G, \mu_G)$ and for  every $d\in \mathbb{N}$, there exists a $\gamma$-equiangular tight continuous modular frame 	$\{\tau_g\}_{g\in G}$  for $\mathcal{A}^d$ such that $\mu_G(G)=d^2$}.
\end{conjecture}
\begin{theorem}
	Let  	$\{\tau_g\}_{g\in G}$ be a unit inner product modular  continuous frame for  $\mathcal{E}$. Then 
	\begin{align}\label{1}
		\mathcal{M}(	\{\tau_g\}_{g\in G})\geq \left(\frac{1}{(\mu_G\times\mu_G)((G\times G)\setminus\Delta)}\left[\frac{\mu_G(G)^2}{d}-(\mu_G\times\mu_G)(\Delta)\right]\right)^\frac{1}{2}\eqqcolon\gamma.
	\end{align}
	If the frame is $\gamma$-equiangular, then we have equality in Inequality (\ref{1}).
\end{theorem}

%Reason is the following. The operator $A$ is constructed as follows. Let $\mathcal{X}$ be a Banach space and assume that $\mathcal{X}=\mathcal{Y} \oplus \mathcal{Z}$ for some closed subspaces $\mathcal{Y}$ and $\mathcal{Z}$ of $\mathcal{X}$. Then the projections $P:\mathcal{Y} \oplus \mathcal{Z}\ni y\oplus z \to y \in \mathcal{X}$, $Q:\mathcal{Y} \oplus \mathcal{Z}\ni y\oplus z \to Z \in \mathcal{X}$  are bounded linear operators. Define $A\coloneqq 2P$ which is linear. Then
%\begin{align*}
%	A-I_\mathcal{X}=2P-(P+Q)=P-Q.
%\end{align*}
%Now for $x=y\oplus z\in \mathcal{Y} \oplus \mathcal{Z}$, we have
%\begin{align*}
%	\|Ax-x\|=\|Px-Qx\|=\|y\oplus z\| \quad \text{ and } \quad \|Ax\|=\|2Px\|=2\|y\|.
%\end{align*}
%Thus there is no $\beta\geq 0$ such that 
%\begin{align*}
%		\|Ax-x\|=\|y\oplus z\|\leq 2\beta\|y\|=\beta  \|Ax\|,\quad \forall x=y\oplus z\in \mathcal{Y} \oplus \mathcal{Z}.
%\end{align*}

 \bibliographystyle{plain}
 \bibliography{reference.bib}

\end{document}